\documentclass[12pt]{article}
\usepackage{amsmath}
  \usepackage{paralist}
   \usepackage{amssymb}
  \usepackage{amsthm}
  \usepackage{graphics} 
  \usepackage{epsfig} 
\usepackage{graphicx}  \usepackage{epstopdf}
 \usepackage[colorlinks=true]{hyperref}
\hypersetup{urlcolor=blue, citecolor=red}

\newtheorem{theorem}{Theorem}[section]

\newtheorem{lemma}[theorem]{Lemma}
\newtheorem{proposition}{Proposition}

\newtheorem{remark}{Remark}

\begin{document}
 \begin{center}      
 ON THE DECOUPLING OF THE IMPROVED BOUSSINESQ EQUATION INTO TWO UNCOUPLED CAMASSA-HOLM EQUATIONS
 \end{center}
 






\centerline{\scshape H.A. Erbay$^*$ and S. Erbay}
\medskip
{\footnotesize
 \centerline{ Department of Natural and Mathematical Sciences, Faculty of Engineering,}
   \centerline{Ozyegin University,}
   \centerline{ Cekmekoy 34794, Istanbul, Turkey}
} 

\medskip

\centerline{\scshape A. Erkip}
\medskip
{\footnotesize
 \centerline{ Faculty of Engineering and Natural Sciences,}
   \centerline{Sabanci University,}
   \centerline{Tuzla 34956,  Istanbul,    Turkey}

}

\bigskip


\begin{abstract}
  We rigorously establish that, in the long-wave regime characterized by the assumptions of long wavelength and small amplitude, bidirectional solutions of the improved Boussinesq equation tend to associated solutions of two uncoupled Camassa-Holm equations. We give a precise estimate for  approximation errors in terms of two small positive parameters measuring the effects of nonlinearity and dispersion. Our results demonstrate that, in the present regime, any solution of the improved Boussinesq equation is split into two waves propagating in opposite directions independently, each of which is governed by the Camassa-Holm equation. We observe that the approximation error for the decoupled problem considered in the present study is greater than the approximation error for the unidirectional problem characterized by a single Camassa-Holm equation.   We also consider lower order approximations and we state similar error estimates for both the Benjamin-Bona-Mahony approximation and the Korteweg-de Vries approximation.
\end{abstract}

\section{Introduction}\label{sec1}

In this study, we consider the improved Boussinesq (IB) equation
\begin{equation}
    u_{tt}-u_{xx}-\delta^{2}u_{xxtt}-\epsilon (u^{2})_{xx}=0,  \label{ib}
\end{equation}
which appears as a relevant model in various areas of physics (see, e.g. \cite{ostrovskii,soerensen,duruk2009} for solid mechanics), and we proceed along our analysis of the Camassa-Holm (CH)  approximation of the IB equation initiated in \cite{erbay2016}. In (\ref{ib}), $u=u(x,t)$ is a  real-valued function, and $\epsilon $ and $\delta $ are two small positive parameters measuring the effects of nonlinearity and dispersion, respectively.  In \cite{erbay2016}, by a proper choice of initial data, we  restricted our attention to the right-going solutions of the IB equation and showed that, for small amplitude long waves, they are well approximated by associated solutions of a single CH equation  \cite{camassa}. In the present study we remove the assumption about the solutions being unidirectional and consider solutions traveling in both directions with  general initial disturbances. We then show that, in the long-wave regime, solutions of the IB equation can be split into two counter-propagating parts up to a small error. To be more precise, it is shown that any solution of the IB equation is  well approximated by the sum $w^{+}+w^{-}$ of solutions of two uncoupled CH equations
\begin{align}
 &  w^{+}_t+ w^{+}_x+ \epsilon w^{+}w^{+}_x -\frac{3}{4}\delta^{2}w^{+}_{xxx}-{\frac{5}{4}}\delta^{2}w^{+}_{xxt}
             -{\frac{3}{4}}\epsilon \delta^{2}(2 w^{+}_x w^{+}_{xx}+ w^{+} w^{+}_{xxx})=0,  \label{ch-r}\\
 &  w^{-}_t- w^{-}_x- \epsilon w^{-} w^{-}_x +\frac{3}{4}\delta^{2} w^{-}_{xxx}-{\frac{5}{4}}\delta^{2}w^{-}_{xxt}
              +{\frac{3}{4}}\epsilon \delta^{2}(2 w^{-}_x w^{-}_{xx}+ w^{-} w^{-}_{xxx})=0,  \label{ch-l}
\end{align}
where $w^{+}$ and $w^{-}$ denote the right and left going waves, respectively.  We mainly establish existence, consistency and convergence results for the CH approximation of the IB equation in the decoupled case. We prove the decomposition and give the convergence rate between bounded solutions of the IB equation and the sum of two counter-propagating solutions of uncoupled CH equations. We observe that the approximation errors remain small in suitable norms over an arbitrarily long time interval.     We also give error estimates for the Benjamin-Bona-Mahony (BBM) and  Korteweg-de Vries (KdV) approximations of the IB equation  in the same setting, where  $w^{+}$ and $w^{-}$ are solutions of the two uncoupled   BBM  equations \cite{bbm} or KdV equations \cite{korteweg}.

The KdV, BBM  and CH  equations arise as formal asymptotic models for unidirectional propagation of weakly nonlinear and weakly dispersive waves in  a variety of physical situations. Recently, there has been a growing interest to rigorously relate solutions of the asymptotic equations to solutions of the parent equations of original physical problem. For instance, in the context of water waves, the KdV, BBM and CH equations have been rigorously justified as  unidirectional asymptotic models of the water wave equations in \cite{craig1985}, \cite{alazman2006} and \cite{constantin2009}, respectively (the reader is referred to \cite{lannes} for a detailed discussion of the water waves problem). In the case of bidirectional propagation of small amplitude long waves, an uncoupled system of two KdV equations, one for waves moving to the left and one for waves moving to the right, appears as the simplest asymptotic model of  the underlying physical problem.  In  \cite{schneider2000}, \cite{bona2005}, \cite{wright2005}, it was proven that bidirectional, small amplitude, long-wave solutions of the water wave problem are well approximated by combinations of solutions of two uncoupled KdV equations.  In \cite{duchene2014}, a similar justification framework was used for an uncoupled system of two CH equations once again in the water wave setting.

In this paper, attention is given to the IB equation that describes the time evolution of nonlinear dispersive waves in many practically important situations. In \cite{schneider1998} and \cite{wayne2002}, the validity of the uncoupled KdV system  was established  as a leading order approximation for long wavelength solutions of the IB equation. In the present work we extend  the analysis to moderate amplitudes by considering an uncoupled system of two CH equations  as a leading order approximation to the IB equation  in the long wave regime and provide an  estimate for the approximation error. As a by-product, we also recover both the uncoupled KdV system and the uncoupled BBM system as the leading approximations of the IB equation. We believe that the study of the IB equation provides a useful step in understanding long-wave limits of the evolution equations modeling much more complicated physical situations. For a mathematical description of the long-wave limit of unidirectional solutions of the IB equation by a single CH equation we refer to \cite{erbay2016} (see \cite{erbay2015} for the formal derivation of the CH equation from the IB equation).  As in  \cite{erbay2016} the methodology used in this study adapts the techniques in \cite{bona2005,constantin2009,gallay2001}. Since the proofs in the present work are somewhat parallel with the proofs in \cite{erbay2016}, we will  present the new ingredients only.

Several points are worth emphasizing briefly. First,   we remind that  the system of uncoupled CH equations (\ref{ch-r}) and (\ref{ch-l}) can be written in a more standard form by means of the following coordinate transformations
\begin{equation}
    \overline{x}={\frac{2}{\sqrt 5}}(x-{\frac{3}{5}}t), \quad \overline{y}={\frac{2}{\sqrt 5}}(x+{\frac{3}{5}}t), \quad \overline{t}={\frac{2}{3 \sqrt 5}}t .
\end{equation}
 Then, (\ref{ch-r}) and (\ref{ch-l}) become
\begin{align}
 &   v^{+}_{\bar{t}}+ {\frac{6}{5}}  v^{+}_{\bar{x}}+ 3\epsilon  v^{+}  v^{+}_{\bar{x}}
        -\delta^2  v^{+}_{\bar{t}\bar{x}\bar{x}} -{\frac{9}{5}}\epsilon \delta^2(2  v^{+}_{\bar{x}}  v^{+}_{\bar{x}\bar{x}}
        +  v^{+}  v^{+}_{\bar{x}\bar{x}\bar{x}})=0,  \label{vr}  \\
 &   v^{-}_{\bar{t}}-{\frac{6}{5}}  v^{-}_{\bar{y}}- 3\epsilon  v^{-}  v^{-}_{\bar{y}}
    -\delta^2  v^{-}_{\bar{t}\bar{y}\bar{y}} +{\frac{9}{5}}\epsilon \delta^2(2  v^{-}_{\bar{y}}  v^{-}_{\bar{y}\bar{y}}
    +  v^{-}  v^{-}_{\bar{y}\bar{y}\bar{y}})=0,  \label{vl}
\end{align}
with $ v^{+}(\overline{x},\overline{t})=w^{+}(x,t)$ and $v^{-}(\overline{y},\overline{t})=w^{-}(x,t)$, respectively. We also remind that,  using  the scaling transformation $V^{+}(X,\tau) =\epsilon v^{+}(\bar{x},\bar{t})$, $~V^{-}(Y,\tau) =\epsilon v^{-}(\bar{y},\bar{t})$, $~(\bar{x},\bar{y})=\delta (X,Y)$, and  $\bar{t}=\delta \tau$, we can rewrite (\ref{vr}) and (\ref{vl})  in more standard forms with no parameters. Secondly, we observe that the approximation error for the decoupled problem considered in the present study is greater than the approximation error for the unidirectional problem characterized by a single CH equation in \cite{erbay2016}. This deterioration is partially related to the error due to approximate splitting of the initial data for the IB equation. Another factor is due to the fact that the interaction of the right-going and the left-going waves  appears to play a major role in the residual term that arises when we plug the solutions of the uncoupled CH equations into the IB equation. We emphasize that the coupled models for which the interaction terms are not supposed to be small,  provide a better description than the decoupled ones over short time scales and that a rigorous justification of this  claim  remains as an open problem.

The remainder of this paper is organized as follows. First, in Section 2, we focus on a description of the problem setting for approximation errors. In Section 3, we conduct a preliminary discussion of uniform estimates for the solutions of the CH equation and we estimate the residual term that arises when we plug the sum of solutions of the uncoupled CH equations into the IB equation.  In Section 4, using the energy estimate based on certain commutator estimates, we obtain the convergence rate between the solutions of the IB equation and the sum of solutions of the uncoupled CH equations. In Section 5 we recover the BBM  and KdV approximations of the IB equation in the decoupled case.

Our notation  for  function spaces is fairly standard. The notation $\| u\|_{L^p}$ denotes the $L^p$ ($1\leq p <\infty$)  norm of $u$ on   $\mathbb{R}$. The symbol $\big\langle u, v\big\rangle$ represents the inner product of $u$ and $v$ in $L^2$.   The notation  $H^{s}=H^s(\mathbb{R})$ denotes the $L^{2}$-based Sobolev space of order $s$ on $\mathbb{R}$, with the norm $\| u\|_{H^{s}}=\big(\int_\mathbb{R} (1+\xi^2)^s |\widehat u(\xi)|^2 d\xi \big)^{1/2}$.  We will drop the symbol $\mathbb{R}$ in $\int_{\mathbb{R}}$.  The symbol $C$ will stand for a generic positive constant. Partial differentiations are denoted by $D_{t}$, $D_{x}$ etc.

\section{Problem Setting for Approximation Errors }\label{sec2}

In this section, we formulate the Cauchy problem for approximation errors. For this aim we first state the following well-posedness result   \cite{molinet2002,duruk2010} for the initial-value problem of the IB equation:
\begin{theorem}\label{theo2.1}
    Let $u_{0},u_{1}\in H^{s}(\mathbb{R})$, $~s>1/2$. Then  for any pair of parameters $\epsilon$ and $\delta$, there is some $T^{\epsilon, \delta}>0$ so that the Cauchy problem for the IB equation (\ref{ib}) with initial data
    \begin{equation}
        u(x,0) =u_{0}(x) ,\quad u_{t}(x,0)=u_{1}(x), \label{IB-in}
    \end{equation}
     has a unique solution $u \in \ C^{2}\big( [0,T^{\epsilon ,\delta }],H^{s}(\mathbb{R})\big)$.
\end{theorem}
The existence time $T^{\epsilon ,\delta }$ above may depend on $\epsilon$ and $\delta$ and it may be chosen arbitrarily large as long as $T^{\epsilon, \delta }<T_{\max }^{\epsilon ,\delta }$ where $T_{\max }^{\epsilon ,\delta }$ is the maximal time. Furthermore, it was shown in \cite{duruk2010} that the existence time, if it is finite,  is determined by the $L^{\infty }$ blow-up condition
\begin{equation*}
   \limsup_{t\rightarrow T_{\max }^{\epsilon ,\delta }} \left\Vert u\left(t\right) \right\Vert _{L^{\infty }}=\infty .
\end{equation*}%
Let $w^{+}$ and $w^{-}$  be two families of solutions  for the Cauchy problems defined for the  CH equations (\ref{ch-r}) and (\ref{ch-l}) with initial values $w^{+}_{0}$ and $w^{-}_{0}$, respectively. Given the initial data $(u_{0}, u_{1})$ for the IB equation, the first question is how to select the corresponding initial data  $(w^{+}_{0}, w^{-}_{0})$ for the  CH equations (\ref{ch-r}) and (\ref{ch-l}). Ideally we should have $u_{0}=w^{+}_{0}+w^{-}_{0}$ and $u_{1}=w^{+}_{t}(x,0)+w^{-}_{t}(x,0)$, yet it will be convenient to choose $(w^{+}_{0}, w^{-}_{0})$ independent of the parameters $\epsilon$ and $\delta$. From the uncoupled CH equations we get
  \begin{equation*}
w^{+}_{t}+w^{-}_{t}=-w^{+}_{x}+w^{-}_{x}+\mathcal{O}(\epsilon,\delta^2,\epsilon \delta^2).
  \end{equation*}
  Neglecting the higher order terms yields the approximation $u_{1}(x)=-w^{+}_{x}(x,0)+w^{-}_{x}(x,0)$. Finally, assuming that $u_{1}=(v_{0})_{x}$ we get
  \begin{equation}
    w^{+}_{0}={1\over 2}(u_{0}-v_{0}), \quad  w^{-}_{0}={1\over 2}(u_{0}+v_{0}).
\end{equation}
Our aim is to compare the solution $u$ of (\ref{ib}) and (\ref{IB-in}) with the sum $w^{+}+ w^{-}$.

Obviously, the error function defined by $r=u-(w^{+}+w^{-})$ satisfies the initial condition $r(x, 0) =0$. In order to express $r_t(x,0)$ in terms of the  initial values $(u_{0}, v_{0})$ of the IB equation, we substitute the  CH equations (\ref{ch-r}) and (\ref{ch-l}) into $r_t(x,0)$:
 \begin{align}
  r_t(x,0)=& u_{1}(x)-\Big(w^{+}_{t}(x,0)+w^{-}_{t}(x,0)\Big) \nonumber \\
          =& -(w_0^+)_x +(w_0^-)_x- (1-{5\over 4} \delta^2 D_x^2)^{-1}\bigg\{-D_x( w_0^+ - w_0^-) \nonumber \\
          &  -{\epsilon\over 2} D_x \Big( (w_0^+)^2 -(w_0^-)^2\Big) +{3\over 4} \delta^2D_x^3( w_0^+ -w_0^-)  \nonumber \\
          &     +{3 \over 4}\epsilon \delta^2 D_x \bigg( {1\over 2} \Big(\big((w_0^+)_x\big)^2 -\big((w_0^-)_x\big)^2\Big)
                      +w_0^+ (w_0^+)_{xx}   -w_0^- (w_0^-)_{xx} \bigg)   \bigg\}\nonumber \\
          =& D_{x}(1-{5\over 4} \delta^2 D_x^2)^{-1}\bigg\{ -{1\over 2}\delta^2(v_0)_{xx} -{1\over 2}\epsilon u_0 v_0  \nonumber \\
          &   -{3\over 8} \epsilon \delta^2  \Big( (u_0)_x (v_0)_x- (u_0 v_0)_{xx}\Big) \bigg\} . \label{err-r}
  \end{align}
Substituting $u=r+w^{+}+w^{-}$ into (\ref{ib}), we observe that the function $r$ satisfies
\begin{equation}
        \left( 1-\delta^{2}D_{x}^{2}\right) r_{tt}-r_{xx}-\epsilon \Big(r^{2}+2(w^{+}+w^{-})r\Big)_{xx}=-\widetilde{F}_{x}, \label{error-equ}
\end{equation}%
where $\widetilde{F}_{x}$ is the residual term given by
\begin{equation}
    \widetilde{F}_x= F_x^{+}+F_x^{-}-2\epsilon \left(w^{+} w^{-}\right)_{xx}, \label{F-def}
\end{equation}%
with
\begin{equation}
    F_x^{\mp}= w^{\mp}_{tt}-w^{\mp}_{xx}-\delta ^{2}w^{\mp}_{xxtt}-\epsilon \Big((w^{\mp})^2\Big)_{xx}. \label{F-pm}
\end{equation}
Our main problem is now reduced to finding an upper bound for $r$ in terms of $\epsilon$ and $\delta$.

We note that $r_{t}(x,0)$ is of the form  $\big(q(x)\big)_{x}$ by (\ref{err-r}). Since $r(x,0)=0$ and the nonhomogeneous term in (\ref{error-equ}) is of the form $-\widetilde{F}_{x}$, one can show  that $r=\rho_{x}$ for some appropriate function $\rho(x,t)$ (see \cite{duruk2010} for the homogeneous case).
To  further simplify the calculations, in what follows we will express (\ref{error-equ}) in terms of  both $\rho$ and $r$ as
\begin{equation}
    \left( 1-\delta ^{2}D_{x}^{2}\right) \rho _{tt}-r_{x}-\epsilon \Big(r^{2}+2(w^{+}+w^{-}) r\Big) _{x}=-\widetilde{F}.\label{rhor}
\end{equation}%
with the initial data
\begin{align}
   & r(x,0)= 0, \\
   &  \rho_t(x,0)=(1-{5\over 4}\delta^2 D_x^2)^{-1}  \bigg\{ -{\delta^2\over 2}  (v_0)_{xx}- {\epsilon \over 2} u_0 v_0
       -{3\over 8}\epsilon \delta^2  \Big((u_0)_x (v_0)_x -(u_0 v_0)_{xx}\Big) \bigg\}.\label{rho-ini}
\end{align}%

\section{Some Estimates for the CH Equation and the Nonhomogeneous IB Equation}\label{sec3}
In this section, we state some previous estimates from \cite{erbay2016} concerning solutions of the  CH equation and the nonhomogeneous IB-type equation. For the convenience of the reader we provide short versions of the proofs in the Appendix.

The following proposition  is a direct consequence of the estimates proved by Constantin and Lannes in \cite{constantin2009} for a more general class of equations, containing the CH equation as a special case. We refer the reader to Section 2 of \cite{erbay2016} for a more detailed discussion. As a result we have the following proposition:

\begin{proposition}[Corollary 2.1 of \cite{erbay2016}] \label{prop1}
         Let $w_{0} \in H^{s+k+1}\left( \mathbb{R}\right) $,  $ s>1/2$, $k\geq 1$. Then, there exist $T>0$,  $C>0$  and a unique family of solutions
        \begin{equation*}
            w^{\epsilon ,\delta }\in C\left( [0,\frac{T}{\epsilon }],H^{s+k}(\mathbb{R})\right) \cap C^{1}\left( [0,\frac{T}{\epsilon }],H^{s+k-1}(\mathbb{R})\right)
        \end{equation*}%
        to the CH equation
        \begin{equation}
            w_t+ w_x+ \epsilon w w_x -\frac{3}{4}\delta^{2} w_{xxx}-{\frac{5}{4}}\delta^{2}w_{xxt} -{\frac{3}{4}}\epsilon \delta^{2}(2 w_x w_{xx}+ w w_{xxx})=0.
            \label{ch}
        \end{equation}
         with initial value  $w(x, 0)=w_{0}(x)$, satisfying
        \begin{equation*}
            \left\Vert w^{\epsilon ,\delta }\left( t\right) \right\Vert _{H^{s+k}}+\left\Vert w_{t}^{\epsilon ,\delta }(t) \right\Vert _{H^{s+k-1}}\leq C,
        \end{equation*}%
        for all $0<\epsilon \leq \delta \leq 1$ and $t\in \lbrack 0,\frac{T}{\epsilon }\rbrack$.
\end{proposition}

Plugging $w$ of (\ref{ch}) in the IB equation we get a residual term $f$,
\begin{equation}
    f=w_{tt}-w_{xx}-\delta ^{2}w_{xxtt}-\epsilon (w^{2})_{xx}.  \label{residual0}
\end{equation}%
Calculation in \cite{erbay2016} shows that $f$ is of the form $f=F_{x}$ where
\begin{align}
F =& \epsilon^2({\frac{w^{3}}{3}})_{x}-{\frac{1}{8}}\epsilon^2 \delta^2\Big( 3(w_{x}^{2}+2ww_{xx})_{x}
         -3w(w^{2})_{xxx}+2w_{xx}(w^{2})_{x}+w_{x}(w^{2})_{xx}\Big)  \nonumber \\
    & + {\frac{1}{16}} \delta^4\Big( (D_{x}D_{t}-3D_{x}^{2})(3w_{xxx}+5w_{xxt})\Big) \nonumber \\
    &   +{\frac{1}{32}}\epsilon \delta^{4}\Big( 3(D_{x}^{2}D_{t}-3D_{x}^{3})(w_{x}^{2}+2ww_{xx})  \nonumber \\
    &  +2(-3wD_{x}^{2}+2w_{xx}+w_{x}D_{x})(3w_{xxx}+5w_{xxt})\Big) \nonumber \\
    & +{\frac{1}{32}}\epsilon^2 \delta^4 \Big( (-9wD_{x}^{3}+6w_{xx}D_{x}+3w_{x}D_{x}^{2})(w_{x}^{2}+2ww_{xx})\Big). \label{uni-F}
\end{align}%
Furthermore, using the uniform bounds in Proposition \ref{prop1}, the following estimate for $F$ was proved in  \cite{erbay2016}:
\begin{lemma}[Lemma 3.1 of \cite{erbay2016}]\label{lem3.1}
    Let $w_{0} \in H^{s+6}\left( \mathbb{R}\right)$, $~s>1/2$ and let $w^{\epsilon, \delta}$ be the  family of solutions  to  the CH equation  (\ref{ch}) with initial value \ $w(x, 0) =w_{0}(x)$. Then, there is some $C>0$ so that the family of residual terms $F=F^{\epsilon, \delta}$ in (\ref{uni-F})  satisfies
\begin{equation*}
    \left\Vert F\left( t\right) \right\Vert _{H^{s}}\leq C\left( \epsilon ^{2}+\delta ^{4}\right) ,
\end{equation*}%
for all $0<\epsilon \leq \delta \leq 1$ and $t\in \lbrack 0,\frac{T}{\epsilon}\rbrack$.
\end{lemma}

We next consider the solution $r,\rho$ of the IB-type equation
\begin{equation}
        \left( 1-\delta ^{2}D_{x}^{2}\right) \rho _{tt}-r_{x}-\epsilon \left(r^{2}+2\widetilde{w}r\right) _{x}=-\widetilde{F},   \label{r-equ}
    \end{equation}%
where $r=\rho_{x}$.  We assume that  $\widetilde{w}$ and $\widetilde{F}$ are given  functions depending on $\epsilon$ and $\delta$, with
\begin{align}
  & \widetilde{w}\in C\left( [0,\frac{T}{\epsilon }],H^{s+1}(\mathbb{R})\right), \\
  & \left\Vert \widetilde{w}(t) \right\Vert_{H^{s+1}}\leq C \quad   \textrm{for} \quad t\in [0, {T\over \epsilon}], \\
  & \widetilde{F}\in C\left( [0,\frac{T}{\epsilon }],H^{s}(\mathbb{R})\right).
 \end{align}
Our purpose is to find a bound for solutions of (\ref{r-equ}). In that respect,  following the approach in \cite{gallay2001} and \cite{erbay2016}, we define the "energy" as
    \begin{align}
        E_{s}^{2}(t)
        =& \frac{1}{2}\Big( \left\Vert \rho_{t}(t) \right\Vert_{H^{s}}^{2}+\delta^{2}\left\Vert r_{t}(t) \right\Vert_{H^{s}}^{2}+\left\Vert r(t) \right\Vert_{H^{s}}^{2}\Big)
        +\epsilon \big\langle \Lambda^{s}\big(\widetilde{w}(t) r(t)\big), \Lambda^{s}r(t) \big\rangle   \nonumber \\
        & +\frac{\epsilon }{2}\big\langle \Lambda^{s}r^{2}(t), \Lambda^{s}r(t) \big\rangle,  \label{ener}
    \end{align}%
where  $\Lambda^{s}=\left(1-D_{x}^{2}\right)^{s/2}$.  Taking the energy in the usual form without the $\epsilon$ terms will yield a loss of  $\delta $ in the final  estimate. This is due to the coefficient $\delta^{2}$ of the term $\left\Vert r_{t}(t) \right\Vert_{H^{s}}^{2}$ (see Remark 2 of \cite{erbay2016} for further details).

Since $w^{+}$ and $w^{-}$ exist for all times $t\leq T/\epsilon$, $r(x,t)$ will exist over the same time interval unless $r$ or equivalently $u^{\epsilon, \delta}$ blows up in a shorter time. By Theorem \ref{theo2.1} the blow-up of $u^{\epsilon, \delta}$ is controlled by the $L^{\infty}$-norm. Thus the blow-up of $r$ is also determined by its $L^{\infty}$-norm or equivalently by $\left\Vert r(t)
        \right\Vert _{H^{s}}$. Since $r(x,0)=0$ we define
    \begin{equation}
        T_{0}^{\epsilon, \delta }
        =\sup \left\{ t\leq {T\over \epsilon}: \left\Vert r(\tau)
        \right\Vert _{H^{s}}\leq 1 \quad \textrm{for all }\quad \tau \in [0, t]\right\} . \label{patlama}
    \end{equation}%
   Note that
    \begin{equation*}
        \left\vert \big\langle \Lambda^{s}(\widetilde{w}r), \Lambda^{s}r\big\rangle\right\vert
        \leq C\left\Vert r(t) \right\Vert_{H^{s}}^{2}, \quad \textrm{and} \quad
        \left\vert \big\langle \Lambda^{s}r^{2}, \Lambda^{s}r\big\rangle\right\vert
        \leq \left\Vert r(t) \right\Vert_{H^{s}}^{3} \leq \left\Vert r(t) \right\Vert_{H^{s}}^{2},
    \end{equation*}
     where we have used (\ref{patlama}) and the uniform estimate for $\widetilde{w}$.  Thus, for sufficiently small values of $\epsilon$ and $t\leq T^{\epsilon, \delta}_{0}$, we have
    \begin{equation*}
        E_{s}^{2}\left( t\right)
            \geq \frac{1}{4}\Big( \left\Vert \rho_{t}(t)\right\Vert_{H^{s}}^{2}+\delta^{2}\left\Vert r_{t}(t)\right\Vert_{H^{s}}^{2}
                +\left\Vert r(t)\right\Vert_{H^{s}}^{2}\Big),
    \end{equation*}%
    which shows that $E_{s}^{2}(t) $ is positive definite. The above result also shows that an estimate obtained for     $E_{s}^{2}(t)$ gives     an estimate for $\left\Vert r(t)\right\Vert_{H^{s}}^{2}$.
    After a series of calculations and estimates we obtain  the differential inequality for the energy:
  \begin{equation}
    {d\over dt} E_s(t) \leq C\left ( \epsilon  E_s(t) + \sup_{t\leq T/\epsilon}\left\Vert \widetilde{F}(t) \right\Vert_{H^{s}} \right )
    \label{energy-ineq}
  \end{equation}
The proofs of Lemma \ref{lem3.1} and this inequality were given in \cite{erbay2016}. We will summarize those proofs in the Appendix for the convenience of the reader.

\section{Convergence Proof for the Decoupled \\ Approximation }\label{sec4}

In this section we prove our main result, Theorem \ref{theo4.1} given below. Recall from Section \ref{sec2} that we started with the family of solutions  $u=u^{\epsilon, \delta}$ of the IB equation and chose appropriate solutions $w^{\mp}$ of the uncoupled CH equations. Our aim is to show that the sum $w^{+}+w^{-}$ is a good approximation for $u$. In other words, we want to find a good estimate for the error, namely the solution of the problem defined by (\ref{rhor})-(\ref{rho-ini}). This is achieved by the proof of Theorem \ref{theo4.1}, where we take advantage of the results in Section \ref{sec3}.
\begin{theorem}\label{theo4.1}
    Let  $u_{0} \in H^{s+6}(\mathbb{R})$ and $v_{0} \in H^{s+7}(\mathbb{R})$, $s>1/2$. Suppose $u^{\epsilon, \delta}$ is the solution of the IB equation (\ref{ib}) with initial data \begin{equation*}
    u(x,0) =u_{0}(x), \quad u_{t}(x,0)=(v_{0}(x))_{x}.
    \end{equation*}
    Let
    \begin{equation*}
    w^{+}_{0}={1\over 2}(u_{0}-v_{0}), \quad w^{-}_{0}={1\over 2}(u_{0}+v_{0}).
    \end{equation*}
    Then, for any given $t^{*}>0$ there exists  $\delta^{*}\leq 1$     so that    the solutions $(w^{\mp})^{\epsilon ,\delta}$ of the uncoupled CH  equations (\ref{ch-r}) and (\ref{ch-l}) with initial values $w^{\mp}(x,0)=w^{\mp}_{0}(x)$   satisfy
    \begin{equation*}
        \Vert u^{\epsilon, \delta}(t)-(w^{+})^{\epsilon, \delta}(t)-(w^{-})^{\epsilon, \delta}(t) \Vert_{H^{s}}\leq
        ~ C\Big( (\epsilon +\delta^2)+(\epsilon +\delta^4)t\Big)
    \end{equation*}
    for all $t\in \left[ 0,t^{*}\right] $ and all $0<\epsilon \leq \delta \leq \delta^{*}$.
\end{theorem}
\begin{proof}
    We first note that (\ref{rhor}) is exactly (\ref{r-equ}) with $\widetilde{w}=w^{+}+w^{-}$ and $\widetilde{F}= F^{+}+F^{-}-2\epsilon \left(w^{+} w^{-}\right)_{x}$. The explicit form of $F^{+}$ is obtained by substituting $w^{+}$  in place of $w$ in (\ref{uni-F}). Similarly, the explicit form of $F^{-}$ is obtained by substituting $w^{-}$ for $w$ and $-t$  for $t$ in (\ref{uni-F}).  Since $w^{+}$ and  $w^{-}$ are  solutions of the CH equations (\ref{ch-r}) and (\ref{ch-l}), by Proposition 1 and Lemma 3.1 we have the estimates
    \begin{equation*}
        \left\Vert w^{\mp}\left( t\right) \right\Vert _{H^{s+1}}\leq C, \qquad \left\Vert F^{\mp}\left( t\right) \right\Vert _{H^{s}}\leq C\left( \epsilon ^{2}+\delta ^{4}\right) ,
    \end{equation*}%
    for all $0<\epsilon \leq \delta \leq 1$ and $t\in \lbrack 0,\frac{T}{\epsilon}\rbrack $. Therefore,
    \begin{align}
        & \left\Vert \widetilde{w}\left( t\right) \right\Vert _{H^{s+1}}\leq \left\Vert w^{+}\left( t\right) \right\Vert _{H^{s+1}}+\left\Vert w^{-}\left( t\right) \right\Vert _{H^{s+1}}\leq C, \\
        & \left\Vert \widetilde{F}\left( t\right) \right\Vert _{H^{s}}\leq
        \left\Vert F^{+}\left( t\right) \right\Vert _{H^{s}}+\left\Vert F^{-}\left( t\right) \right\Vert _{H^{s}}
        +2\epsilon \left\Vert (w^{+}w^{-})_{x}(t) \right\Vert_{H^{s}}\leq C\left( \epsilon+\delta ^{4}\right).
    \end{align}%
    Then (\ref{energy-ineq}) becomes
    \begin{equation}
    {d\over dt} E_s(t) \leq C\Big( \epsilon  E_s(t) + ( \epsilon+\delta ^{4} ) \Big)
      \end{equation}
      implying
      \begin{align*}
    E_{s}(t)\leq E_{s}(0)e^{C\epsilon t}+ \frac{\epsilon+\delta^4}{\epsilon }\left( e^{C\epsilon t}-1\right). \label{en-result}
    \end{align*}
    We have $r(x,0)=0$. Since the operator $(1-{5\over 4}\delta^{2}D_{x}^{2})^{-1}$ is bounded on $H^{s}$, from (\ref{err-r}) and (\ref{rho-ini}) we get $ \left\Vert r_{t}(0) \right\Vert _{H^{s}}\leq C \left( \epsilon+\delta ^{2} \right )$ and $ \left\Vert \rho_{t}(0) \right\Vert _{H^{s}}\leq C \left( \epsilon+\delta ^{2} \right )$. By (\ref{ener}) we get
    \begin{equation}
    E_{s}(0)\leq C \left( \epsilon+\delta ^{2} \right ).
    \end{equation}
    Thus
    \begin{align*}
        E_{s}(t)\leq C\left( \epsilon+\delta ^{2} \right )e^{C\epsilon t} + \frac{\epsilon+\delta^4}{\epsilon }\left( e^{C\epsilon t}-1\right),
    \end{align*}
    or
    \begin{align}
        E_{s}(t)\leq C\Big( (\epsilon +\delta^2)+(\epsilon +\delta^4)t\Big).  \label{en-result}
    \end{align}
    We note that this estimate holds for all $t\leq T_{0}^{\epsilon, \delta}\leq T/\epsilon$, namely, as long as $\left\Vert r(t) \right\Vert _{H^{s}}\leq 1$. Given any $t^{*} >0$ we have  $t^{*}\leq T/\epsilon$ for sufficiently small $\epsilon$. Then we can find some $\delta^{*}$ such that for all $\epsilon \leq \delta \leq \delta^{*}\leq 1$ and $C\Big( (\epsilon +\delta^2)+(\epsilon +\delta^4)t^{*}\Big)\leq 1$. By (\ref{en-result}) we will get $\left\Vert r(t) \right\Vert _{H^{s}}\leq 1$ for all $t\leq t^{*}$, which means the estimate above holds for all $t\leq t^{*}$.
\end{proof}

We want to conclude this section with some remarks about the above theorem.
\begin{remark}
    We observe that the error involves two parts. The constant term in (\ref{en-result}) is due to the approximation error in splitting the initial data of the IB equation, while the term $\epsilon t$ arises from the interaction term $\epsilon w^{+}w^{-}$ in (\ref{F-def}).
\end{remark}
\begin{remark}
         The  error for the unidirectional CH approximation of the IB equation   was obtained in \cite{erbay2016} as $\mathcal{O}\Big((\epsilon^{2} +\delta^4)t\Big)$ for times of order $\mathcal{O}  (1/ \epsilon)$. Comparing with the estimate in Theorem \ref{theo4.1}, we observe that the single CH equation provides a better  approximation.
\end{remark}

\section{The BBM and KdV Approximations}\label{sec5}

In this section we consider the BBM  and the KdV approximations of the IB equation in the decoupled case.  The analysis is similar in spirit to that of Sections \ref{sec3} and \ref{sec4}. Recall that the main ingredients  were the uniform estimate Proposition \ref{prop1}, the residual estimate in Lemma \ref{lem3.1} for the CH equation and the energy estimate (\ref{energy-ineq}) for the IB-type equation. The uniform estimate for the BBM equation can be obtained from \cite{constantin2009} whereas for the uniform estimates of the KdV equation we refer to \cite{alazman2006}.  The next step is to calculate the corresponding residual terms $\widetilde{F}= F^{+}+F^{-}-2\epsilon \left(w^{+} w^{-}\right)_{x}$.  As the details can be found in \cite{erbay2016}, we therefore give only the final results.

\subsection{The BBM Approximation}

For the BBM approximation we obtain $F^{\mp}$ as
\begin{align*}
    F^{+}&=\epsilon^{2} \left({(w^{+})^{3}\over 3}\right)_{x}-{1\over 4}\epsilon\delta^{2}\left(6w^{+}w^{+}_{xxt}
    +2w^{+}_{x}w^{+}_{xt}+w^{+}_{t}w^{+}_{xx}-9w^{+}_{x}w^{+}_{xx}\right) \nonumber \\
    &     +\frac{1}{16}\delta^{4}D_{x}^{3}\left(5w^{+}_{tt}-12w^{+}_{xt}-9w^{+}_{xx}\right)\nonumber \\
    F^{-}&=\epsilon^{2} \left({(w^{-})^{3}\over 3}\right)_{x}+{1\over 4}\epsilon\delta^{2}\left(6w^{-}w^{-}_{xxt}
    +2w^{-}_{x}w^{-}_{xt}+w^{-}_{t}w^{-}_{xx}+9w^{-}_{x}w^{-}_{xx}\right) \nonumber \\
    &     +\frac{1}{16}\delta^{4}D_{x}^{3}\left(5w^{-}_{tt}+12w^{-}_{xt}-9w^{-}_{xx}\right).
\end{align*}%
 Using the energy inequality (\ref{energy-ineq}) and making a similar argument, we obtain the BBM version of Theorem  \ref{theo4.1}.
 \begin{theorem}\label{theo5.1}
    Let  $u_{0} \in H^{s+6}(\mathbb{R})$ and $v_{0} \in H^{s+7}(\mathbb{R})$, $s>1/2$. Suppose $u^{\epsilon, \delta}$ is the solution of the IB equation (\ref{ib}) with initial data \begin{equation*}
    u(x,0) =u_{0}(x), \quad u_{t}(x,0)=\big(v_{0}(x)\big)_{x}.
    \end{equation*}
    Let
    \begin{equation*}
    w^{+}_{0}={1\over 2}(u_{0}-v_{0}), \quad w^{-}_{0}={1\over 2}(u_{0}+v_{0}).
    \end{equation*}
    Then, for any given $t^{*}>0$ there exists   $\delta^{*}\leq 1$       so that    the solutions $(w^{\mp})^{\epsilon ,\delta}$ of the uncoupled BBM  equations
     \begin{align}
        &    w^{+}_t+ w^{+}_x+ \epsilon w^{+} w^{+}_x -\frac{3}{4}\delta^{2} w^{+}_{xxx}-{\frac{5}{4}}\delta^{2}w^{+}_{xxt}=0,
                \label{bbm-r}\\
        &    w^{-}_t- w^{-}_x- \epsilon w^{-} w^{-}_x +\frac{3}{4}\delta^{2} w^{-}_{xxx}-{\frac{5}{4}}\delta^{2}w^{-}_{xxt}=0
                \label{bbm-l}
    \end{align}
     with initial values $w^{\mp}(x,0)=w^{\mp}_{0}(x)$   satisfy
    \begin{equation*}
        \Vert u^{\epsilon, \delta}(t)-(w^{+})^{\epsilon, \delta}(t)-(w^{-})^{\epsilon, \delta}(t) \Vert_{H^{s}}\leq
        ~C\Big( (\epsilon +\delta^2) + (\epsilon +\delta^4) t \Big)
    \end{equation*}
    for all $t\in \left[ 0,t^{*}\right] $ and all $0<\epsilon \leq \delta \leq \delta^{*}$.
\end{theorem}

\subsection{The KdV Approximation}

For the KdV approximation we obtain $F^{\mp}$ as
\begin{align}
   & F^{+}= D_{x}\bigg\{ \frac{1}{3}\epsilon^{2} (w^{+})^3
        +\frac{1}{4}\epsilon\delta^{2}\Big(-3(w^{+}_{x})^2+4(w^{+}w^{+}_{x})_{t}\Big)
        +\frac{1}{4}\delta^{4}(-w^{+}_{xxxx}+2w^{+}_{xxxt})\bigg\} \nonumber \\
    & F^{-}= D_{x}\bigg\{\frac{1}{3}\epsilon^{2}(w^{-})^{3}
        +\frac{1}{4}\epsilon\delta^{2}\Big(-3(w^{-}_{x})^{2}-4(w^{-}w^{-}_{x})_{t}\Big)
        +\frac{1}{4}\delta^{4}(-w^{-}_{xxxx}-2w^{-}_{xxxt})\bigg\}. \nonumber \\ \label{ff}
\end{align}
We note that these residual terms contain higher-order derivatives compared to those of the CH and BBM approximations. This is reflected in the higher smoothness requirements for the initial data in the following theorem
\begin{theorem}\label{theo5.3}
    Let  $u_{0} \in H^{s+7}(\mathbb{R})$ and $v_{0} \in H^{s+8}(\mathbb{R})$, $s>1/2$. Suppose $u^{\epsilon, \delta}$ is the solution of the IB equation (\ref{ib}) with initial data
    \begin{equation*}
    u(x,0) =u_{0}(x), \quad  u_{t}(x,0)=\big(v_{0}(x)\big)_{x}.
    \end{equation*}
    Let
    \begin{equation*}
    w^{+}_{0}={1\over 2}(u_{0}-v_{0}), \quad  w^{-}_{0}={1\over 2}(u_{0}+v_{0}).
    \end{equation*}
    Then, for any given $t^{*}>0$  and  $0<c_{1}<c_{2}$   there exists     $\delta^{*}\leq 1/ \sqrt{3}$ such that  the solutions $(w^{\mp})^{\epsilon ,\delta}$ of the uncoupled KdV  equations
     \begin{align}
    & w^{+}_{t}+w^{+}_{x}+\epsilon w^{+}w^{+}_{x}+{\frac{\delta ^{2}}{2}}w^{+}_{xxx}=0,  \label{kdv-r} \\
    & w^{-}_{t}-w^{-}_{x}-\epsilon w^{-}w^{-}_{x}-{\frac{\delta ^{2}}{2}}w^{-}_{xxx}=0  \label{kdv-l}
    \end{align}
    with initial values $w^{\mp}(x,0)=w^{\mp}_{0}(x)$   satisfy
    \begin{equation*}
        \Vert u^{\epsilon, \delta}(t) -(w^{+})^{\epsilon, \delta}(t) -(w^{-})^{\epsilon, \delta}(t)\Vert _{H^{s}}\leq ~C \epsilon  (1+t)
    \end{equation*}%
    for all $\ t\in \left[ 0, t^{*}\right] $ and  all $0<\delta \leq \delta^{*}$,
    $\epsilon \in \left[\frac{\delta^{2}}{c_{2}}, \frac{\delta^{2}}{c_{1}}\right]$.
\end{theorem}
The error estimates for the BBM and KdV approximations are of the same order as the CH approximation but in the case of the KdV approximation this is only valid for $\epsilon \approx \delta^{2}$. This is due to the fact that the KdV equation arises  in the long-wave regime defined by the balance between  dispersive and nonlinear effects.

\section{Appendix}
In this appendix, we provide  short versions of the proofs of Lemma \ref{lem3.1},  the differential inequality  given by (\ref{energy-ineq}) and the commutator estimates used. For more details, we refer the reader to \cite{erbay2016}.

\subsection{Proof of Lemma 3.1}
\begin{proof}
Except for the term $D_{x}^{3}D_{t}^{2}w$,  $F$ is a combination of terms of the form $D_{x}^{j}w$ with  $j\leq 5$ or $D_{x}^{l}D_{t}w$ with  $l\leq 4$.  Using the CH equation (\ref{ch}) the term $D_{x}^{3}D_{t}^{2}w$ can be written as
\begin{align*}
    D_{x}^{3}D^{2}_{t}w
    = -\mathcal{Q}D_{x}^{3}D_{t}\left(w_{x}+\epsilon ww_{x}\right)
         +\frac{3}{4}\delta^{2} \mathcal{Q} D_{x}^{6}D_{t}w
         +\frac{3}{4}\epsilon \delta^{2} \mathcal{Q}D_{x}^{3}D_{t} (2w_{x}w_{xx}+ww_{xxx}),
\end{align*}%
where the operator $\mathcal{Q}$ is
\begin{equation}
   \mathcal{Q}=\left( 1-\frac{5}{4}\delta^{2}D_{x}^{2}\right)^{-1}.  \label{qoperator}
\end{equation}
The operator norms of $\mathcal{Q}$ and $\mathcal{Q}\delta ^{2}D_{x}^{2}$  are bounded on $H^{s}$:
\begin{equation*}
    \left\Vert \mathcal{Q}\right\Vert_{H^{s}}\leq 1   \quad \textrm{and} \quad
        \left\Vert \delta^{2}\mathcal{Q}D_{x}^{2}\right\Vert_{H^{s}}\leq \frac{4}{5}.
\end{equation*}%
And the rest of the terms are again of the form $D_{x}^{j}w$ with  $j\leq 5$ or $D_{x}^{l}D_{t}w$ with  $l\leq 4$. Taking care of the coefficients  $\epsilon^{2}$,  $\epsilon^{2} \delta^{2}$,  $\delta^{4}$,  $\epsilon \delta^{4}$ or $\epsilon^{2} \delta^{4}$, we obtain the following  estimate
\begin{equation}
    \left\Vert F(t) \right\Vert _{H^{s}}\leq C \left( \epsilon^{2}+\delta ^{4}\right)
    \Big( \left\Vert w\right\Vert_{H^{s+5}}  +\left\Vert w_{t}\right\Vert _{H^{s+4}}\Big).  \label{saa}
\end{equation}%
Using Proposition \ref{prop1} with $k=5$, we complete the proof.
\end{proof}

\subsection{Proof of the Energy Inequality (\ref{energy-ineq})}

\begin{proof}
    Below we  will use  the following commutator estimates:
    \begin{equation}
    \big\langle \lbrack \Lambda^{s}, w\rbrack g,\Lambda^{s}h\big\rangle \leq
    C\Vert w\Vert_{H^{s+1}}\Vert g \Vert_{H^{s-1}}\Vert h \Vert_{H^{s}},
     \label{est1}
\end{equation}%
 and
\begin{equation}
    \big\langle \Lambda \lbrack \Lambda ^{s},w\rbrack h,\Lambda ^{s-1}g\big\rangle \leq C\Vert w\Vert_{H^{s+1}}\Vert h \Vert_{H^{s}}\Vert g\Vert_{H^{s-1}},  \label{est2}
\end{equation}%
where $\lbrack \Lambda^{s}, w\rbrack =\Lambda^{s}w-w\Lambda^{s}$. These estimates are particular cases of the general estimates given in  Proposition B.8 of \cite{lannes} (see also \cite{erbay2016} for further details).

We now differentiate $E_{s}^{2}(t)$ with respect to $t$ and then eliminate the term $\rho_{tt}$ from the resulting equation using (\ref{rhor}). Thus we have
    \begin{align}
        \frac{d}{dt}E_{s}^{2}(t)
        =&\frac{d}{dt}\Big( \epsilon \big\langle \Lambda^{s}(\widetilde{w}r),\Lambda^{s}r\big\rangle
                    +\frac{\epsilon }{2}\big\langle \Lambda^{s}r^{2},\Lambda^{s}r\big\rangle\Big)\nonumber \\
        &           -\epsilon \big\langle \Lambda^{s}(r^{2}+2\widetilde{w}r),\Lambda^{s}r_{t}\big\rangle
                    -\big\langle \Lambda^{s}\widetilde{F},\Lambda^{s}\rho_{t}\big\rangle  \nonumber \\
        =&\epsilon \Big( \big\langle \Lambda^{s}(\widetilde{w}_t r),\Lambda^{s}r\big\rangle
                    -\big\langle \Lambda^{s}(\widetilde{w}r),\Lambda^{s}r_{t}\big\rangle
                    +\big\langle \Lambda^{s}r,\Lambda^{s}(\widetilde{w}r_{t})\big\rangle \nonumber \\
         &     + \big\langle \Lambda^{s}(rr_{t}),\Lambda^{s}r\big\rangle
             -\frac{1}{2}\big\langle \Lambda^{s} r^{2}, \Lambda^{s} r_{t}\big\rangle \Big)
                  -\big\langle \Lambda^{s}\widetilde{F},\Lambda^{s}\rho_{t}\big\rangle . \label{decom}
    \end{align}%
    The estimates for the  first term in the parentheses  and the last term  are
    \begin{align*}
         & \big\langle \Lambda^{s}(\widetilde{w}_t r),\Lambda^{s}r\big\rangle
            \leq C\left\Vert r\right\Vert_{H^{s}}^{2}\leq C E_{s}^{2}  \\
        & \big\langle \Lambda^{s}\widetilde{F}, \Lambda^{s}\rho_{t}\big\rangle
            \leq  \sup_{t\leq T/\epsilon}\Vert \widetilde{F}(t)\Vert_{H^{s}}\Vert \rho_{t}\Vert_{H^{s}}
            \leq \sup_{t\leq T/\epsilon}\Vert \widetilde{F}(t)\Vert_{H^{s}}E_{s},
    \end{align*}
    respectively. We rewrite the second and the third terms in the parentheses  in (\ref{decom})  in the form
    \begin{equation}
        -\big\langle \Lambda^{s}(\widetilde{w}r),\Lambda^{s}r_{t} \big\rangle
        +\big\langle \Lambda^{s}r,\Lambda^{s} (\widetilde{w}r_{t})\big\rangle
       =-\big\langle \lbrack \Lambda^{s},\widetilde{w}\rbrack r,\Lambda^{s}r_{t}\big\rangle
            +\big\langle \lbrack \Lambda^{s},\widetilde{w}\rbrack  r_{t},\Lambda^{s}r\big\rangle. \label{liz}
    \end{equation}%
     Then,  using the commutator estimates (\ref{est1})-(\ref{est2}) we estimate  the two terms on the right-hand side of (\ref{liz}) as
    \begin{align}
        \big\langle \lbrack \Lambda^{s},\widetilde{w}\rbrack  r,\Lambda ^{s}r_{t}\big\rangle
        =&\big\langle \Lambda \lbrack \Lambda ^{s},\widetilde{w}\rbrack  r,\Lambda ^{s-1}r_{t}\big\rangle
        \leq  C\Vert \widetilde{w}\Vert _{H^{s+1}}\Vert r\Vert _{H^{s}}\Vert r_{t}\Vert _{H^{s-1}},  \label{comw2} \\
        \big\langle \lbrack \Lambda ^{s},\widetilde{w}\rbrack   r_{t},\Lambda ^{s}r\big\rangle
        \leq &C\Vert \widetilde{w}\Vert _{H^{s+1}}\Vert r\Vert _{H^{s}}\Vert r_{t}\Vert _{H^{s-1}}.%
    \end{align}%
    The fourth and fifth terms in  the parentheses  in (\ref{decom}) can be written as
    \begin{align}
        \big\langle \Lambda^{s}(rr_{t}),\Lambda^{s}r\big\rangle -& \frac{1}{2}
        \big\langle \Lambda^{s}r^{2},\Lambda^{s}r_{t}\big\rangle \nonumber \\
       =&\big\langle \Lambda^{s-1} ( 1-D_{x}^{2})r,\Lambda^{s-1} (rr_{t})\big\rangle
           -\frac{1}{2}\big\langle \Lambda^{s-1} (1-D_{x}^{2}) r^{2},\Lambda^{s-1}r_{t}\big\rangle \nonumber \\
       =& \big\langle \Lambda^{s-1}r,\Lambda^{s-1}(rr_{t})\big\rangle
           -\frac{1}{2}\big\langle \Lambda^{s-1} (r^{2}-2r_{x}^{2}),\Lambda^{s-1}r_{t}\big\rangle \nonumber \\
       & -\Big( \big\langle \Lambda^{s-1}(rr_{t}),\Lambda^{s-1}r_{xx}\big\rangle
           -\big\langle \Lambda^{s-1}r_{t}, \Lambda^{s-1} (rr_{xx})\big\rangle\Big). \label{fourth}
   \end{align}%
    If we group the first two terms together  in the above equation, we get the following estimate
    \begin{align*}
        \left\vert \big\langle \Lambda^{s-1}r,\Lambda^{s-1}(rr_{t})\big\rangle
            -\frac{1}{2}\big\langle \Lambda^{s-1}(r^{2}-2r_{x}^{2}),\Lambda^{s-1}r_{t}\big\rangle\right\vert
        \leq &  C\Vert r\Vert _{H^{s-1}}^{2}\Vert r_{t}\Vert _{H^{s-1}}, \\
         \leq & C\Vert r\Vert _{H^{s}}^{2}\Vert r_{t}\Vert _{H^{s-1}}.
    \end{align*}%
    Similarly, if we group the last two terms  in (\ref{fourth}) together, we obtain the  estimate
    \begin{align*}
        \left\vert \big\langle \Lambda^{s-1}(rr_{t}),\Lambda^{s-1}r_{xx}\big\rangle
            -\big\langle \Lambda^{s-1}r_{t},\Lambda^{s-1}(rr_{xx})\big\rangle\right\vert
        \leq &C\Vert r\Vert _{H^{s}}\Vert r_{t}\Vert _{H^{s-1}}\Vert r_{xx}\Vert _{H^{s-2}} \\
        \leq &C\Vert r\Vert _{H^{s}}^{2}\Vert r_{t}\Vert _{H^{s-1}},
    \end{align*}%
    which follows from (\ref{liz}) and ( \ref{comw2}) if $\widetilde{w}$, $r$, $r_{t}$ are replaced, respectively, by
    $r$, $r_{t}$, $r_{xx}$ and $s$ by $s-1$.  Also, we remind that
    \begin{equation*}
        \Vert r_{t}\Vert _{H^{s-1}}
        =\Vert \rho _{xt}\Vert _{H^{s-1}}\leq \Vert \rho_{t}\Vert _{H^{s}}\leq C E_{s}
    \end{equation*}%
    and   $\Vert r\Vert _{H^{s}}\leq 1$.  Combining all the above results we obtain from (\ref{decom}) that
    \begin{equation*}
       \frac{d}{dt}E_{s}^{2}(t)
            \leq C\bigg( \epsilon E_{s}^{2}(t) +\Big( \sup_{t\leq T/\epsilon}\left\Vert \widetilde{F}(t)\right\Vert_{H^{s}} \Big)E_{s}(t) \bigg),
   \end{equation*}%
    which reduces to (\ref{energy-ineq}) if we cancel $E_{s}(t)$ from both sides of the equation.
\end{proof}


\end{document}